\RequirePackage{fix-cm}
\documentclass[12pt, reqno]{amsart}
\usepackage{graphicx}
\usepackage{amsmath, amsxtra, amssymb, latexsym}
\usepackage{color}
\usepackage[active]{srcltx}
\usepackage[mathscr]{eucal}
\usepackage{amsmath,amssymb}

\newtheorem{theorem}{Theorem}[section]

\newcommand{\comments}[1]{}

\renewcommand{\phi}{\varphi}
\newcommand{\R}{\mathbb R}

\newcommand{\nn}{\nonumber}

\newcommand{\ra}{\rightarrow}

\newcommand{\h}{\mathbb H}

\newcommand{\dt}{\frac{d}{dt}}

\newcommand{\lra}{\longrightarrow}
\newcommand{\lan}{\langle}
\newcommand{\ran}{\rangle}
\newcommand{\han}{\hat{a}^{(n)}}
\newcommand{\hbn}{\hat{b}^{(n)}}
\newcommand{\hun}{\hat{u}^{(n)}}
\newcommand{\hvn}{\hat{v}^{(n)}}
\newcommand{\hvnn}{\hat{v}^{(n-1)}}
\newcommand{\hann}{\hat{a}^{(n-1)}}
\newcommand{\hunn}{\hat{u}^{(n-1)}}
\newcommand{\hbnn}{\hat{b}^{(n-1)}}
\newcommand{\tu}{\widetilde{\bar{u}}}

\newcommand{\tb}{\widetilde{\bar{b}}}

\def\intav#1{\mathchoice
          {\mathop{\vrule width 6pt height 3 pt depth -2.5pt
                  \kern -9pt \intop}\nolimits_{\kern -6pt#1}}%
          {\mathop{\vrule width 5pt height 3 pt depth -2.6pt
                  \kern -6pt \intop}\nolimits_{#1}}%
          {\mathop{\vrule width 5pt height 3 pt depth -2.6pt
                  \kern -6pt \intop}\nolimits_{#1}}%
          {\mathop{\vrule width 5pt height 3 pt depth -2.6pt
                  \kern -6pt \intop}\nolimits_{#1}}}

\newcommand{\charfn}[1]{{\raisebox{1.2pt}{\mbox{$\chi
_{\kern-1pt\lower3pt\hbox{{$\scriptstyle{#1}$}}}$}}}}

  \def \R{\mathcal R}

\def\<{\langle} \def\>{\rangle}
\begin{document}
\title[Cross-Chemotaxis Model]{Cross-Chemotaxis System Derived from an Atherosclerotic Plaque Model}

\author{Jonathan Bell}
\address{Department of Mathematics \& Statistics, 
University of Maryland Baltimore County,
1000 Hilltop Circle, Baltimore, MD-21250, USA.}
\email{jbell@umbc.edu}
\author{Animikh Biswas$^{\dagger}$}
\address{Department of Mathematics \& Statistics, 
University of Maryland Baltimore County,
1000 Hilltop Circle, Baltimore, MD-21250, USA.}
\email{abiswas@umbc.edu}
 \thanks{${}^{\dagger}$ \textit{Corresponding author}}
 \thanks{For A.B., this research was partly supported by NSF grant DMS15-17027. For J.B., travel assistance was provided by the Simons Foundation grant 283689 }

\maketitle
\begin{abstract}
In modeling the inflammatory response to a lesion in an artery wall, there are a number of chemotactic mechanisms going on within the wall layer that lead to an arterial plaque. We introduce a rather reduced model of these dynamic processes, but the focus of this paper is a subsystem of the full model system that has independent interest. Namely, the system here consists of two cell densities, each producing a chemical that is a chemoattractant for the other cell type. Then we prove positivity, local and global existence, and discuss some qualitative behavior of solutions. 
\end{abstract}

\section{Introduction}
The following investigation came out of a project to model the development of an atherosclerotic plaque, and specifically to follow the development of the plaque's fibrous cap and its possible degradation. This led to a sub model of the form
\begin{equation} \label{sys}
\left \{ \begin{array}{l}
u_{t}=D_{1}u_{xx}-\alpha_{1}(ua_{x})_{x} - \mu u \\ \\
v_{t}=D_{2}v_{xx}-\alpha_{2}(vb_{x})_{x} + \rho(b)f(v) \\ \\
a_{t}=D_{3}a_{xx} + \beta_{1}v - \mu_{a}a \\ \\
b_{t}=D_{4}b_{xx} + \beta_{2}u - \mu_{b}b
\end{array} \right.
\end{equation}
along with initial conditions
\begin{equation} \label{ics}
\mbox{At $t=0:$} \quad \quad u=u_{0}(x)\;,\; v=v_{0}(x)\;,\; a=a_{0}(x)\;,\;b=b_{0}(x) \;\;.
\end{equation}  
Here $u$ and $v$ represent two different types of cell densities within an arterial cell layer, the intima, which we denote rather arbitrarily by $0 < x < 1$. Also, $u_{0} \geq 0$, $v_{0} \geq 0$, $a_{0} \geq 0$, $b_{0} \geq 0$ on domain $x \in [0,1]$. So system (\ref{sys}) is defined on domain $\Omega_{T}=\{ (x,t): 0<x<1\;,\;0<t<T \}$. Cell type $v$ (respectively, $u$) produces a chemoattractant chemical $a$ (resp. $b$) for cell type $u$ (resp. $v$), with chemoattractant affinities (positive constant) parameters being $\alpha_{1}$ and $\alpha_{2}$. All the theoretical work the authors are aware of, whether in one or higher space dimensions, has involved a single cell type producing the chemoattractant concentration (``self-chemotaxis''). A feature not usually included in chemotaxis models is the addition of a death term for $u$, death rate $\mu > 0$, and a proliferation term for $v$, namely a (Verhulst) growth rate, where the intrinsic growth parameter $\rho$ depends on the chemoattractant $b$. 

For the boundary conditions, let 
\begin{equation} \label{bcs} \begin{array}{l}
\mbox{At $x=0$: $a_{x}=b_{x}=0$\;;\; so $u_{x}=0$ and $-v_{x}=h(b(0,t))$} \\ \\
\mbox{At $x=1$: $a_{x}=b_{x}=0$\;;\; so $u_{x}=g(t)$ and $v_{x}=0$ \quad   $0 \leq t \leq T$}
\end{array}
\end{equation}

More details on the functions $\rho, f, h,$ and $g$ will be given below. The original goal in considering this system was to understand if the modeling allowed the two cell densities to concentrate near $x=1$ where the fibrous cap, which separates the plaque from the blood, develops. Figure 1 gives an example case of a numerical solution that does just that. 

There is a large and growing literature on the analysis of chemotaxis systems. A couple of good review articles, though slightly dated, are Horstmann[8] and Hillen and Painter [5]. See also the books of Perthame [15] and Suzuki [16]. Global existence for 1D models was accomplished by Osaki and Yagi [12]. For higher space dimensions, various global results are in, for example, Horstmann[7], Gajewski and Zacharias [4], Biler [3], Wrzosek [18], and Hillen and Potapov [6]. These contributions all involve a single cell population and single chemoattractant, and often have homogeneous Neumann or Dirichlet boundary conditions. A notable exception is the early work of Lauffenburger, et al [10], which has a kinetic term on the cell concentration equation, and nonlinear boundary conditions. While most chemotaxis models do not include death or kinetic terms for cell populations, a few that do include Painter and Hillen [14], Little, et al [11], Osaki, et al [13], and Wrzosek [18], and the previously mentioned work of Lauffenburger, et al [10] (and follow-on work of Wang [17] and Zeng [19]). Our model here, besides having two cell types and two ``cross-chemoattractants'',  has a nonlinear kinetic term and nonlinear boundary conditions.  Our system can be put into Amann's `separated divergence form' (see Amann [1]); so a local (smooth) solution can be obtained this way, but we chose an alternative fixed-point approach. Our existence results are closest in flavor to that of Hillen and Potapov [6], though our analysis differs significantly. Also, our system seems to have most of the characteristics of Amann's normally elliptic systems (Amann [2]), but we elected not to pursue that approach to global existence.

In the next section we mention some background leading to problem (\ref{sys})-(\ref{bcs}). The sections after that will concern analysis of the cross-chemotaxis model (\ref{sys}), and a future paper will consider the analysis of the full plaque model. Section 3 is concerned with local existence of a solution, while section 4 deals with positivity of solutions. Section 5 develops global bounds on the solution and hence a global solution, then we discuss more formal qualitative solution behavior in section 6. The final section will discuss possible extensions and their relevance to other physical problems.

\section{Outline of Model Development}
Arteries are layered structures with the inner basal membrane contacting the blood made of endothelial cells, followed outwardly by the intima layer, then the media layer, and finally external layers of more structural material. An arterial plaque is a lesion that develops in the intima layer. It is made up of immune cells, cell debris, lipids, fibrous connective tissue, etc. Arterial plaque formation and growth involves complex chemical, hemodynamic, and biomechanics processes. There are basically two types of plaques. Stable plaques that mainly cause no detrimental clinical effects, unless it becomes sufficiently large to compromise blood flow. Then there are unstable, or vulnerable, plaques that are high risk of rupturing and causing thrombosis. The latter plaques generally have larger lipid cores and thinner fibrous caps than stable plaques. We are interested in why some plaques develop into vulnerable plaques, and the modeling has concentrated on plaque processes that couple to cap dynamics. Plaque vulnerability is tied directly to degradation of the fibrous cap, and we believe a one space dimensional model is sufficient to address this issue.  

Our biological model considers some principal chemical processes, neglecting the hemodynamic and most of the biomechanical processes, the endothelial layer mechanisms, plaque growth aspects, and the complex three-dimensional geometry. It primarily investigates mechanisms for migration of smooth muscle cells and immune cells (macrophages) toward the endothelial layer, which provide critical extracellular matrix material, mainly collagen, to build the cap, and chemicals that tend to degrade the cap. 

Chemoattractant activity is important to this fibrous cap development. There is production of macrophage chemoattractants, like chemotactic peptide-1, from smooth muscle cells, and production of growth factors, like PDGF-B, from macrophages that play a role in smooth muscle cell migration and proliferation within the plaque. A full discussion of the model development, with analysis,  will be done elsewhere. Our interest here is the analysis of the above (sub)model involving two cell types, macrophage population represented by the scaled $u$ variable, and smooth muscle cells, represented by $v$, and two chemoattractants, $a$ and $b$. 

\section{Existence of a Local Solution}
We will now state the assumptions.
\begin{enumerate}
\item All parameters $D_i, \alpha_i, \beta_i, \mu, \mu_a, \mu_b >0$.
\item $\rho:\R \lra \R_+$  is continuous, bounded, and uniformly Lipschitz on bounded sets.
\item $f:\R \lra \R$ such that $f$ is uniformly Lipschitz on compact sets, $f(v) \le 0$ for $v \ge K$, and $f(v) \ge 0\ \forall\ v <K$. The number $K$ is referred to as the carrying capacity of the population. 
\item $h:\R \lra \R_+, h(0)=0$ and $h$ is bounded.
\item $g:\R_+ \lra \R_+, g(0)=0$ and $\int_0^T g^2(s)ds < \infty $ for all $T>0$.
\end{enumerate}
{\bf Notation:} The norm $\|\cdot\|$ will always denote the $L^2$ norm with respect to the space variable
$x$ . Also, $\lan \cdot, \cdot \ran$ will be the inner product. Furthermore, in all the computations below, $C$ will denote a generic, time independent constant, which may depend on the parameters indicated above but not on the initial data, and whose value may change from one line of the computation to the next.
\begin{theorem}
There exists a local (in time) weak solution of the above system on $[0,T]$ for some $T>0$ with the following regularity properties:
\begin{align*}
& u,v, a, b \in C([0,T]; L^2) \cap L^2 ([0,T]; \h^1)\ \mbox{and additionally,} \\
& a,b \in L^\infty([0,T]; \h^1) \cap L^2 ([0,T]; \h^2).
\end{align*}
\end{theorem}
\begin{proof}
Given $(u^{(n-1)}, v^{(n-1)}, a^{(n-1)}, b^{(n-1)})$ define $(a^{(n)},b^{(n)})$ by
\begin{gather}   \label{anbneq}
a_t^{(n)}=D_3a^{(n)}_{xx} + \beta_1 v^{(n-1)}-\mu_aa^{(n)} \  \mbox{and}\\
b_t^{(n)}=D_4b^{(n)}_{xx} + \beta_2 u^{(n-1)}-\mu_bb^{(n)},
\end{gather}
with homogeneous Neumann boundary conditions, and define 
\begin{gather}  \label{unvneq}
u_t^{(n)} =D_1u^{(n)}_{xx} -\alpha_1 (u^{(n)}a^{(n)}_x)_x -\mu u^{(n)}\ \mbox{and}\\
v_t^{(n)} =D_2v^{(n)}_{xx} -\alpha_2 (v^{(n)}b^{(n)}_x)_x +\rho(b^{(n)})f(v^{(n-1)}).
\end{gather}
Let 
\[
K = 4\max\{\|a_0\|_{\h^1}^2, \|b_0\|_{\h^1}^2, \|u_0\|^2, \|v_0\|^2\}.
\]
Assume that
\begin{gather}  \label{intbd}
\begin{array}{l}
 \sup_{0 < t \le T}\max\left\{\|a^{(n-1)}\|_{\h^1}^2, \|b^{(n-1)}\|_{\h^1}^2, \|u^{(n-1)}\|^2, \|v^{(n-1)}\|^2\right\} \le K, \\
 \\
  \max\left\{ \frac{D_3}{4} \int_0^T \|a^{(n-1)}_{xx}\|^2, 
 \frac{D_4}{4} \int_0^T \|b^{(n-1)}_{xx}\|^2, \frac{D_1}{4} \int_0^T \|u^{(n-1)}_{x}\|^2,
 \frac{D_2}{4} \int_0^T \|v^{(n-1)}_{x}\|^2\right\}\\
  \le K.
 \end{array}
 \end{gather}
 We will prove that the same bounds hold for $(u^{(n)},v^{(n)},a^{(n)},b^{(n)})$.  Taking inner product with $a^{(n)}_{xx}$ in \eqref{anbneq} and integrating by parts, we obtain
 \[
 \frac12 \dt \|a_x^{(n)}\|^2 + D_3 \|a^{(n)}_{xx}\|^2 = -\beta_1 \lan v^{(n-1)},a_{xx}^{(n)}\ran - \mu_a \|a_x^{(n)}\|^2,
 \]
  whence, using Young's inequality, we get
 \begin{gather}  \label{anineq}
 \frac12  \dt  \|a_x^{(n)}\|^2 + \frac{D_3}{2} \|a^{(n)}_{xx}\|^2 \le C\|v^{(n-1)}\|^2-\mu_a\|a_x^{(n)}\|^2.
 \end{gather}
 From \eqref{anineq}, we obtain
 \[
 \frac12 \dt  \|a_x^{(n)}\|^2 +\mu_a\|a_x^{(n)}\|^2 \le C \|v^{(n-1)}\|^2.
 \]
 Thus, by Gronwall's inequality, we arrive at
 \begin{gather*}   
 \|a_x^{(n)}\|^2 \le e^{-\mu_a T}\|a_0\|_{\h^1}^2 + C\int_0^T e^{-\mu_a(t-s)}\|v^{(n-1)}\|^2 \\
 \le e^{-\mu_a T}\|a_0\|_{\h^1}^2 + CTK,
 \end{gather*}
 where in the last inequality, we used \eqref{intbd}. Provided $T$ is sufficiently small
 (namely, $T \le \frac{4}{C}$), we have $\|a_x^{(n)}\|^2 \le K$.
 Inserting this bound in \eqref{anineq} and integrating, we get
 \[
 \|a_x^{(n)}\|^2 + \frac{D_3}{2} \int_0^T \|a_{xx}^{(n)}\|^2 
 \le C \int_0^T \|v^{(n-1)}\|^2 
 \le CTK  \le K,
 \]
 provided $T$ is sufficiently small, which in turn also implies 
 \[
 \frac{D_3}{2} \int_0^T \|a_{xx}^{(n)}\|^2  \le K.
 \]
 The computation for $b^{(n)}$ is similar.
 
 We will now obtain analogous estimates for $u^{(n)}, v^{(n)}$. Taking $L^2$-inner product of \eqref{unvneq} with $u^{(n)}$ and integrating by parts we obtain
 \begin{gather}   \label{unest1}
 \frac12 \dt \|u^{(n)}\|^2 + D_1 \|u_x^{(n)}\|^2 = \alpha_1 \int u^{(n)}a_x^{(n)}u_x^{(n)} - \mu\|u^{(n)}\|^2
 + D_1g(t)u^{(n)}(1).
 \end{gather}
  Note that 
 \[
 D_1|g(t)u^{(n)}(1)| \le D_1g(t)\|u^{(n)}\|_{L^\infty} \lesssim g(t)\|u_x^{(n)}\|
 \le Cg^2(t) + \frac{D_1}{4} \|u_x^{(n)}\|^2,
 \]
 where $x \lesssim y $ denotes $x \le Cy$ for an adequate  constant $C>0$. The last inequality in the line above is obtained using Young's inequality where  a suitable adjustment of constant is made in the application of Young's inequality so that the term $\|u_x^{(n)}\|^2$ has $\frac{D_1}{4}$ in front.  
 Using the inequality $\|\Psi\|_{L^\infty} \lesssim \|\Psi\|^{1/2} \|\Psi_x\|^{1/2} $ followed by Young's inequality, we get
 \begin{gather*}
 |\int u^{(n)}a_x^{(n)}u_x^{(n)}| \le \|a_x^{(n)}\|\|u_x^{(n)}\| \|u^{(n)}\|^{1/2}\|u_x^{(n)}\|^{1/2} \\
 \le \frac{D_1}{4}\|u_x^{(n)}\|^2 + C \|a_x^{(n)}\|^4\|u^{(n)}\|^2,
 \end{gather*}
 where we   used Young's inequality  $ab \le \frac{a^p}{p}+\frac{b^q}{q}, 1/p+1/q=1$ with $p= 4/3, q=4$.
 Using the above two estimates in \eqref{unest1}, we obtain
 \begin{align}
\frac12  \dt \|u^{(n)}\|^2 + \frac{D_1}{2} \|u_x^{(n)}\|^2 &  \le C g^2(t) + C \|a_x^{(n)}\|^4 \|u^{(n)}\|^2 \nn \\
& \le Cg^2(t)  + CK^2 \|u^{(n)}\|^2.  \label{unest2}
\end{align}
First dropping the term $\frac{D_1}{2} \|u_x^{(n)}\|^2$ from the inequality \eqref{unest2} and applying Gronwall's inequality, we obtain
\[
\|u^{(n)}\|^2 \le e^{CK^2T} \left\{\|u_0\|^2 + \int_0^T g^2 (s)\, ds \right\} \le K,
\]
where the last inequality holds provided $T$ is sufficiently small, where the size of $T \le T'$ depends only on $K$ and $\int_0^{T'} g^2(s)\, ds$ (To see this, note that due to $\int_0^{T'}g^2(s)ds < \infty$, by DCT, 
$\lim_{T \ra 0} \int_0^T g^2(s)\, ds=0$. Recall that $\|u_0\|^2 \le \frac{K}{4}$. Choose $T$ sufficiently small so that $e^{CK^2T} \le 2$ and $\int_0^{T'}g^2(s)ds \le \frac{K}{8}$ ). 
Now inserting this bound for $\|u^{(n)}\|^2$ in \eqref{unest2} and integrating again we obtain
\[
\frac12 \|u^{(n)}\|^2 + \frac{D_1}{2}  \int_0^T \|u_x^{(n)}\|^2 
\le \frac12 \|u_0\|^2+ C\int_0^T g^2 (s) + CK^3 T \le K,
\]
provided $T$ is adequately small. Once again, the  size of $T$ that is permissible depends only
on $K$ and $\int_0^{T'} g^2(s)$. In other words, provided $T\le T'$ is adequately small
(depending only on the initial data through $K$ and $\int_0^{T'} g^2(s)$), the estimate
\eqref{intbd} is satisfied for all $n$. The estimate for $v^{(n)}$ follows in a similar manner.

Next we show that $(a^{(n)}, b^{(n)}, u^{(n)},v^{(n)})$ forms a convergent sequence with respect to 
the norm
\[
\|(a^{(n)}, b^{(n)}, u^{(n)},v^{(n)})\|= \max\{\|a_x^{(n)}\|, \|b_x^{(n)}\|, \|u^{(n)}\|, \|v^{(n)}\|\}; 
\]
in fact it forms a contractive sequence for $T$ adequately small. To that end, denote $\han = a^{(n)} - a^{(n-1)}$ and similarly for $b, u$ and $v$.

From \eqref{anbneq}, it follows that
\[
\han_t = D_3 \han_{x} + \beta_1 \hvnn -\mu_a \han.
\]
As before, we obtain 
\begin{align*}
& \frac12 \dt \|\han_x\|^2 + D_3 \|\han_{xx}\|^2 = \beta_1 \lan\hvnn,\han_{xx}\ran + \mu_a \lan \han,\han_{xx}\ran \\
& \le C\|\hvnn\|^2 + \frac{D_3}{2}\|\han_{xx}\|^2 + \mu_a\|\han_x\|^2,
\end{align*}
which in turn implies
\[
 \|\han_x\|^2 \le CTe^{2\mu_aT}\sup_{0 \le s \le T}\|\hvnn(s)\|^2.
\]
Thus, $\{a^{(n)}\}$ forms a contractive sequence in $\h^1$ provided $T$ is small (i.e. provided,
for instance, if $CTe^{2\mu_aT} \le \frac14$).
The proof for $\{b^{(n)}\}$ is similar.

Consider now $\hun$. We have 
\[
\hun_t = D_1\hun_{xx} -\alpha_1(\hun a^{(n)}_x)_x - \alpha_1(u^{(n-1)}\han_x)_x -\mu \hun.
\]
As before, we obtain
\begin{gather}  \label{udiffeq}
\frac12 \dt \|\hun\|^2 +D_1\|\hun_x\|^2 + \mu \|\hun\|^2
= \alpha_1 \lan u^{(n-1)}\han_x,\hun_x\ran \\ \notag+  \alpha_1 \lan \hun a_x^{(n)},\hun_x\ran.
\end{gather}
Note that
\begin{align}
\alpha_1 |\lan u^{(n-1)}\han_x,\hun_x\ran|
& \le \alpha_1 \|\han\|\|\hun\|\|u^{(n-1)}\|^{1/2}\|u_x^{(n-1)}\|^{1/2}  \nn \\
& \le \frac{D_1}{4} \|\hun_x\|^2 + C\|u^{(n-1)}\|\|u_x^{(n-1)}\|\|\hann_x\|^2. \label{nonlinu1}
\end{align}
Similarly,
\begin{gather}  \label{nonlinu2}
\alpha_1|\lan \hun a_x^{(n)},\hun_x\ran|
\le \frac{D_1}{4}\|\hun_x\|^2 + C \|\hun\|^2\|a_x^{(n)}\|\|a_{xx}^{(n)}\|.
\end{gather}
Inserting the estimates \eqref{nonlinu1} and \eqref{nonlinu2} in \eqref{udiffeq} and recalling that
$\hun_0=0$, we obtain using Gronwall inequality 
\[
\|\hun\|^2 \le C\left( e^{\int_0^TC \|a_x^{(n)}\|\|a_{xx}^{(n)}\|} \int_0^T \|u^{(n-1)}\|\|u_x^{(n-1)}\|\right)
\sup_{0 \le t \le T}\|\hann_x\|^2.
\]
In view of  \eqref{intbd},
\[
\int_0^T \|a_x^{(n)}\|\|a_{xx}^{(n)}\| 
\le CK\sqrt{T}, \int_0^T \|u^{(n-1)}\|\|u_x^{(n-1)}\| \le CK\sqrt{T}.
\]
Thus, provided $T$ is small (the size of $T$ may depend on $K$ but not on $n$), we have
\begin{equation} \label{finaludiff}
\|\hun\|^2 \le \frac14 \|(\hann,\hbnn,\hunn,\hvnn)\|^2.
\end{equation}

A similar estimate holds also for $\hvn$ where we also need to make use of the fact that
$\rho $ and $f$ are uniformly Lipschitz. To see this, note that as in \eqref{udiffeq}, we have
\begin{align}  
& \frac12 \dt \|\hvn\|^2 +D_2\|\hvn_x\|^2 = \alpha_1 \lan v^{(n-1)}\hbn_x,\hvn_x\ran \nn \\
&  + \alpha_1 \lan \hvn b_x^{(n)},\hvn_x\ran+
\lan \rho(b^{(n)})f(v^{(n-1)})-\rho(b^{(n-1)})f(v^{(n-2)}), \hvn \ran. \label{vdiffeq}
\end{align}
Note that
\begin{align*}
&  \rho(b^{(n)})f(v^{(n-1)})-\rho(b^{(n-1)})f(v^{(n-2)})  = \\
& \rho(b^{(n)})(f(v^{(n-1)}) - f(v^{(n-2)})) 
 + f(v^{(n-2)})(\rho(b^{(n)}) - \rho(b^{(n-1)})).
\end{align*}
Using the global Lipschitz property and uniform bound on $\rho, f$, we have
\[
\| \rho(b^{(n)})f(v^{(n-1)})-\rho(b^{(n-1)})f(v^{(n-2)}) \|
\lesssim \|\rho\|_\infty \|\hvnn\| + \|f\|_\infty \|\hbn\|.
\]
Using this observation and following calculations similar to the ones which yield \eqref{finaludiff} starting from the equality in \eqref{udiffeq}, 
we arrive at
\[
\|\hvn\|^2 \le \frac14 \|(\hann,\hbnn,\hunn,\hvnn)\|^2,
\]
starting from \eqref{vdiffeq}.

Thus, we have the final estimate
\[
\|(\han,\hbn,\hun,\hvn)\| \le \frac{1}{2} \|(\hann,\hbnn,\hunn,\hvnn)\|.
\]

\end{proof}
\section{Positivity of Solutions on $\Omega_{T}$}
For a function $\Psi$, define $\Psi_- = \min\{\Psi, 0\}$.
Moreover,  note that if $\partial \Psi  $ denotes a weak partial derivative of $\Psi $, then   
 $\partial \Psi_- = \left\{ \begin{array}{l}
                                                                                          \partial \Psi,\  \mbox{if}\ \Psi < 0 \\
                                                                                           0,\  \mbox{if}\ \Psi \ge  0
                                                                                           \end{array}\right.$ . 
 \subsection{Positivity of $u$}                                                                                         
In what follows, we will omit the limits of integration as it will be inferred from context. All integrals with respect to the space variable will be on the interval $[0,1]$. First observe that
\begin{equation}  \label{imprel}
\left.
\begin{array}{l}
\int u_t u_- = \int \partial_t(u_-) u_- = \frac12 \frac{d}{dt}\|u_-\|^2
\ \mbox{and}\\
  \int u_x \partial_x(u_-) = \int \left(\partial_x(u_-)\right)^2.
\end{array}
\right\}
\end{equation}
We use equivalent expressions to these for $v, a, b$ below. Taking $L^2-$ inner product of the $u$ equation in (\ref{sys}) with $u_-$, integrating by parts and using \eqref{imprel}, we get
\begin{align*}
 \frac12 \dt \|u_-\|^2 & = -D_1 \|(u_-)_x\|^2 + D_1u_xu_-|_{x=0}^{x=1} - \mu \|u_-\|^2  \\
& \quad -\alpha_1 ua_xu_-|_{x=0}^{x=1} + \alpha_1 \int ua_x (u_-)_x.
\end{align*}
Using the boundary conditions \eqref{bcs}, we have $ \alpha_1 ua_xu_-|_{x=0}^{x=1}=0$ and 
\[
u_xu_-|_{x=0}^{x=1}=g(t)u_- = \left\{ \begin{array}{l}
                                                                           0, \  \mbox{if}\ u(1) \ge 0 \quad \quad (\mbox{since}\, u_-(1)=0)\\
                                                                            <0,\  \mbox{if}\  u(1) < 0 \quad (\mbox{}\, g(t)>0, u_-(1)=u(1)<0).
                                                                            \end{array}\right.
\]
Thus, 
\[
ua_xu_-|_{x=0}^{x=1}=0\ \mbox{and}\ u_xu_-|_{x=0}^{x=1} \le 0. 
\]
Using  this together with the fact  $\int ua_x(u_-)_x = \int (u_-)a_x(u_-)_x$, we obtain
\begin{align*}
\frac12 \dt \|u_-\|^2 & \le -D_1 \|(u_-)_x\|^2 - \mu \|u_-\|^2 + \alpha_1 \int (u_-)a_x(u_-)_x \\
& \le -D_1 \|(u_-)_x\|^2 - \mu \|u_-\|^2 + \alpha_1 \|a_x\|_{L^\infty} \|u_-\|\|(u_-)_x\| \\
& \le - \frac{D_1}{2}  \|(u_-)_x\|^2 - \mu \|u_-\|^2 + C \|a_x\|_{L^\infty}^2 \|u_-\|^2,
\end{align*}
where, to obtain the last inequality, we used Young's inequality 
\begin{align*}
& (\|(u_-)_x\|)\alpha_1 \|a_x\|_{L^\infty} \|u_-\| = (\sqrt{D_1}\|(u_-)_x\|)(\frac{1}{\sqrt{D_1}}
\alpha_1\|a_x\|_{L^\infty} \|u_-\|) \\
& \le \frac{D_1}{2}\|(u_-)_x\|^2 + 
C_{\alpha_1, D_1}  \|a_x\|_{L^\infty}^2 \|u_-\|^2.
\end{align*}
Note now that in space dimension one, we have the inequality $\|a_x\|_{L^\infty} \lesssim \|a\|_{\h^2}$. Moreover, by the regularity property of the local solutions, $\int_0^T \|a\|_{\h^2}^2 ds <  \infty $. Using Gronwall inequality, we have
\[
\|u_-\|^2 \le \|(u_0)_-\|^2 e^{C \int_0^t \|a\|_{\h^2}^2}.
\]
It readily follows that $u \ge 0$ for all $t>0$ provided $u_0 \ge 0$ (or equivalently, $(u_0)_ -=0$).

\subsection{Positivity of $v$}  
Taking $L^2-$ inner product of the $v$ equation in (\ref{sys}) with $v_-$,  and integrating by parts,
a similar calculation as in the previous case yields
\begin{align*}
\frac12 \dt \|v_-\|^2 &  = -D_2\|(v_-)_x\|^2 + D_2 v_xv_-|_{x=0}^{x=1} \\
& \quad + \alpha_2 \int (v_-)b_x (v_-)_x   + \int \rho(b)f(v)v_-\,.
\end{align*}
Note that due to the fact that $h \ge 0$, we have
\[
v_xv_-|_{x=0}^{x=1}= \left\{ \begin{array}{ll}
                                                  0, & v(0) \ge 0 \\
                                                  h(b(0,t))v(0),  & v(0) <0 
                                                  \end{array}\right. \quad \le 0.
\]
Moreover, since $\rho \ge 0$ and $f(v) \ge 0$ for $v \le K$, we have $\int \rho (b)f(v)v_- \le 0$.
Thus,
\begin{align*}
\frac12 \dt \|v_-\|^2 & \le  -D_2\|(v_-)_x\|^2 + \alpha_2 \int (v_-)b_x (v_-)_x \\
 & \le -D_2\|(v_-)_x\|^2 + \alpha_2 \|b_x\|_{L^\infty}\|v_-\|\|(v_-)_x\| \\
 & \le - \frac{D_2}{2} \|(v_-)_x\|^2 + C \|b_x\|_{L^\infty}^2\|v_-\|^2.
 \end{align*}
 Using $\|b_x\|_{L^\infty} \lesssim \|b\|_{\h^2}$ and Gronwall inequality, we deduce as before
 \[
 \|v_-\|^2 \le \|(v_0)_-\|^2 e^{C \int_0^t \|b\|_{\h^2}^2}.
\]
It readily follows that $v \ge 0$ for all $t>0$ provided $v_0 \ge 0$ (or equivalently, $(v_0)_- =0$).

\subsection{Positivity of $a, b$}
From the $b$ equation in (\ref{sys}) we get     
\[
\int b_t b_- = D_4\int b_{xx}b_- + \beta_2 \int ub_- - \mu_b \int b b_-.
\]    
Integrating by parts and using the equivalent expressions to (\ref{imprel}) for $b$, we obtain
\[
\frac12 \dt \|b_-\|^2 = -D_4 \|(b_-)_x\|^2 + b_xb_-|_{x=0}^{x=1} - \mu_b\|b_-\|^2 + \beta_2 \int ub_-.
\]      
Note that the homogeneous Neumann boundary condition for $b$ implies that       
$b_xb_-|_{x=0}^{x=1} =0$. Furthermore, we have shown that $u \ge 0$ if $u_0 \ge 0$, which in turn yields
$\int ub_- \le 0$.      Consequently,
\[
\frac12 \dt \|b_-\|^2 \le -D_4 \|(b_-)_x\|^2 - \mu_b\|b_-\|^2 \le 0.
\]
In particular, this implies that if $b_0 \ge 0\ a.e.$, (equivalently,  $(b_0)_-=0\ a.e.$), then 
$b(x,t) \ge 0\ a.e.$.       Proceeding in a similar manner and using the fact that $v \ge 0\ a.e.$ we can deduce that $a \ge 0\ a.e.$      

\section{Global Bounds and Global Solutions on $\Omega_{T}$}
We will need the following Gagliardo-Nirenberg inequalities to proceed:
\begin{subequations}  \label{gn}
\begin{align}
\|\Psi \|_{L^4} & \lesssim \|\Psi\|_{\h^1}^{1/4}\|\Psi \|^{3/4},   \label{gna} \\
\|\Psi \|_{L^4} & \lesssim \|\Psi \|_{\h^1}^{1/2}\|\Psi\|_{L^1}^{1/2}, \label{gnb}\\
\|\Psi \|_{L^\infty}& \lesssim \|\Psi\|_{\h^1}.  \label{gnc}
\end{align}
\end{subequations}
Proceeding as before, and using \eqref{bcs} and the inequality \eqref{gnc}, we get
\begin{align}
\frac12 \dt \|u\|^2 +D_1\|u_x\|^2 & = D_1 g(t)u(1) -\mu\|u\|^2 + \alpha_1 \int ua_xu_x \nn \\
& \le D_1 g(t)\|u\|_{L^\infty} -\mu \|u\|^2 + \alpha_1 \|u_x\|\|a_x\|_{L^4} \|u\|_{L^4} \nn \\
& \le Cg^2(t) + \frac{D_1}{2} \|u_x\|^2 + C \|a_{x}\|_{L^4}^2\|u\|_{L^4}^2 \nn \\
& \le  Cg^2(t) + \frac{D_1}{2} \|u_x\|^2 + C\|u\|_{L^1}\|u_x\|\|a_{xx}\|^{1/2}\|a_x\|^{3/2}, \label{crineq1}
\end{align}
where in the last line we used \eqref{gnb} to bound $\|u\|_{L^4}$ and \eqref{gna} to bound 
$\|a_x\|_{L^4}$. Using Young's inequality, we obtain
\[
C\|u\|_{L^1}\|u_x\|\|a_{xx}\|^{1/2}\|a_x\|^{3/2}
\le \frac{D_1}{4}\|u_x\|^2 + \frac{D_3}{4} \|a_{xx}\|^2 + C \|u\|_{L^1}^4\|a_x\|^6.
\]
Consequently, from the inequality above and \eqref{crineq1}, we obtain
\begin{gather}  \label{crineq2}
\frac12 \dt \|u\|^2 + \frac{D_1}{4} \|u_x\|^2 \le Cg^2(t) + \frac{D_3}{4}\|a_{xx}\|^2+ C \|u\|_{L^1}^4\|a_x\|^6.
\end{gather}
Proceeding in a similar manner, from the $b$ equation  we obtain
\begin{align*}
\frac12 \dt \|v\|^2 + D_2 \|v_x\|^2 & \le h(b(0,t))v(0)+ \alpha_2 \int vb_xv_x + \|\rho\|_\infty \|f\|_\infty \|v\|_{L^1}\\
& \le \|h\|_\infty \|v_x\| + \|\rho\|_\infty \|f\|_\infty \|v\|_{L^1} + \alpha_2|\int vb_xv_x| \\
& \le C \|h\|_\infty^2 + \frac{3D_2}{4}\|v_x\|^2 + \|\rho\|_\infty \|f\|_\infty \|v\|_{L^1} \\
& + \frac{D_4}{4}\|b_{xx}\|^2
+ C\|v\|_{L^1}^4\|b_x\|^6,
\end{align*}
where in the very last line, we also used Young's inequality.
Thus, we get
\begin{equation} \label{crineq3}
\left.
\begin{array} {lcl}
\frac12 \dt \|v\|^2 + \frac{D_2}{4} \|v_x\|^2
 & \le & C \|h\|_\infty^2+ \|\rho\|_\infty \|f\|_\infty \|v\|_{L^1}\\
 & & \  + \frac{D_4}{4}\|b_{xx}\|^2 
+ C\|v\|_{L^1}^4\|b_x\|^6. 
\end{array}
\right\}
\end{equation}
Differentiating the $a$ and $b$ equations with regard to $x$ yields 
\begin{subequations}
\begin{align}  
a_{xt} &= D_3 a_{xxx}+\beta_1v_x - \mu_a a_x \label{eq:derivchem-a} \\
b_{xt}&= D_4b_{xxx} +\beta_2u_x -\mu_b b_x.  \label{eq:derivchem-b}
\end{align}
\end{subequations}
Taking $L^2$ inner product of \eqref{eq:derivchem-a} with $a_x$, integrating by parts 
in the integral $\int a_{xxx}a_x$ and using the boundary conditions $a_x=0$ for $x=0,1$, we get
\begin{align}
\frac12 \dt \|a_x\|^2 + D_3\|a_{xx}\|^2&  \le \beta_1 \|v_x\|\|a_x\| -\mu_a\|a_x\|^2 \nn \\
&  \le \beta_1 \|v_x\|\|a_x\| \le \frac{D_2}{8} \|v_x\|^2 + C \|a_x\|^2, \label{crineq4}
\end{align}
where to obtain the very last inequality, we used Young's inequality
\[
 \|v_x\|\beta_1\|a_x\|= (\frac{\sqrt{D_2}}{2}\|v_x\|)(\beta_1\frac{2}{\sqrt D_2}\|a_x\|)
 \le \frac{D_2}{8}\|v_x\|^2 + C_{\beta_1,D_2}\|a_x\|^2.
 \]
Similarly from \eqref{eq:derivchem-b} we obtain
\begin{gather} \label{crineq5}
\frac12 \dt \|b_x\|^2 + D_4\|b_{xx}\|^2 \le  \frac{D_1}{8} \|u_x\|^2 + C\|b_x\|^2.
\end{gather}
Let $D = \min\{D_1, D_2, D_3, D_4\}$ and add inequalities \eqref{crineq2}, \eqref{crineq3}, \eqref{crineq4}, \eqref{crineq5} to obtain 
\begin{align}
& \frac12 \dt (\|u\|^2+\|v\|^2 + \|a_x\|^2+\|b_x\|^2)+\frac{D}{8} (\|u_x\|^2+\|v_x\|^2+\|a_{xx}\|^2+\|b_{xx}\|^2) 
\nn \\
& \lesssim (g^2(t)+ \|h\|_{\infty}^2)+\|\rho\|_\infty \|f\|_\infty \|v\|_{L^1}+(\|a_x\|^2+\|b_x\|^2) \nn \\
& + \|u\|_{L^1}^4\|a_x\|^6 + \|v\|_{L^1}^4\|b_x\|^6.  \label{crineq6}
\end{align}
Thus, global existence will immediately follow from \eqref{crineq6} provided one can show, for every $T>0$,  
\begin{subequations} \label{bds}
\begin{align}  
& \sup_{t \in [0,T]}\max\{ \|u(t)\|_{L^1},  \|v(t)\|_{L^1}, \|a(t)\|_{L^1}, \|b(t)\|_{L^1}\}\le C_T< \infty ,  \label{L1bd}\\
& \sup_{t \in [0,T]}\max\{ \|a_x(t)\|, \|b_x(t)\|\}\le \tilde{C}_T< \infty,   \label{h1bd}
\end{align}
where the constants $C_T, \tilde{C}_T$ may depend on the initial data, the parameters, and $T$, but are independent of the solution.
\end{subequations}
Indeed,
assume that $[0,T)$ is the maximal interval of existence of weak solutions with the regularity properties mentioned in the previous section with $T<\infty$. We will establish that in such a case, \eqref{L1bd} holds. Thus, if we can show that \eqref{h1bd} also holds, then the solution can be continued and there is no blow-up, i.e. $T=\infty$. In other words, if \eqref{h1bd} holds, then the solution can be continued further contradicting the fact that $[0,T)$ is the maximal interval of existence. Our next goal is to establish \eqref{bds}.

We will first  establish \eqref{L1bd}. Note that due to the non-negativity of $u,v,a,b$,  $\|u\|_{L^1}=\int udx$; a similar relation holds for each $v, a, b$. Denote $\bar u = \int_0^1 u=\|u\|_{L^1}$, where the last equality holds due to non-negativity of $u$.
By taking inner product with the constant function $1$ and integrating by parts in the $u$ equation, and using the boundary conditions, we have
\[
\dt \bar u = D_1u_x|_{x=0}^{x=1}-\mu \bar u=D_1g(t)-\mu\bar u.
\]
It follows immediately that 
\[
\|u(t)\|_{L^1} \le e^{-\mu t}\|u_0\|_{L^1} +D_1 \int_0^t e^{-\mu(t-s)}g(s)ds,
\]
and \eqref{L1bd} immediately follows concerning $u$.

Concerning $b$, proceeding in a similar manner, we have the equation
\[
\dt \bar b = \beta_2\bar u - \mu_b \bar b.
\]
Using the uniform estimate for $\|u\|_{L^1} $ just obtained, it readily follows that 
\[
\sup_{s \in [0,T)}\|b\|_{L^1} \le C_T < \infty. 
\]
Similarly, we can deduce \eqref{L1bd} for $v, a$.

We now proceed to prove the uniform $\h^1$ bound for $(u,v,a,b)$ asserted in \eqref{h1bd}.
Denote the diffusion semigroup on $[0,1]$ with homogeneous Neumann boundary condition as $S(t)$. A well-known result (see, e.g. [6]) we have is
\[
S(t):L^1 \lra W^{\sigma,r}, \|S(t)\| \lesssim \frac{1}{t^{\gamma_4}}\ \mbox{where}\  \gamma_4 = \frac{\sigma}{2} - \frac{1}{2r}+\frac12.
\]
Thus,
\begin{gather}  \label{opnorm}
S(t):L^1 \lra \h^1=W^{1,2},\quad  \|S(t)\|_{L^1 \ra \h^1} \lesssim \frac{1}{t^{3/4}}.
\end{gather}
Denote $F(s) = \beta_1 v(s) - \mu_a a(s)$ and note that due to \eqref{L1bd},
\[
\sup_{s \in [0,T]} \|F(s)\|_{L^1} < \infty  
\]
 and 
\[
a(t) = S(t)a_0 + \int_0^t S(t-s)F(s)ds.
\]
Using \eqref{opnorm} we immediately obtain
\[
\|a(t)\|_{\h^1} \lesssim \|S(t)a_0\|_{\h^1} + \int_0^t \frac{ds}{(t-s)^{3/4}}\sup_{s \in [0,T]}\|F(s)\|_{L^1}
< \infty ,
\]
where we have also used the fact that  $\|S(t)a_0\|_{\h^1} \le C \|a_0\|_{\h^1}$, where $C$ is independent of $T$.

\comments{
Thus,
\[
\frac12 \dt \|u\|^2 \lesssim g^2(t)+ \|a\|_{\h^2}^2\|u\|^2.
\]
By Gronwall inequality, we obtain
\begin{gather}  \label{ubd}
\|u(t)\|^2 \le e^{C \int_0^t \|a\|_{\h^2}^2 ds}\left[ \|u_0\|^2 + C \int_0^t g^2(s)\right].
\end{gather}
On the other hand, it also follows from \eqref{fullineq} that
\[
 \dt \|u\|^2 +D_1\|u_x\|^2 \le C  \|a\|_{\h^2}^2\|u\|^2.
 \]
 Integrating this we also obtain the estimate
 \[
 D_1 \int_0^t \|u_x\|^2 \le \|u_0\|^2 + \left(\sup_{s \in [0,t]}\|u(s)\|^2\right)\int_0^t \|a\|_{\h^2}^2.
 \]
}
\begin{figure}
\centering
\includegraphics[width=1.0\textwidth]{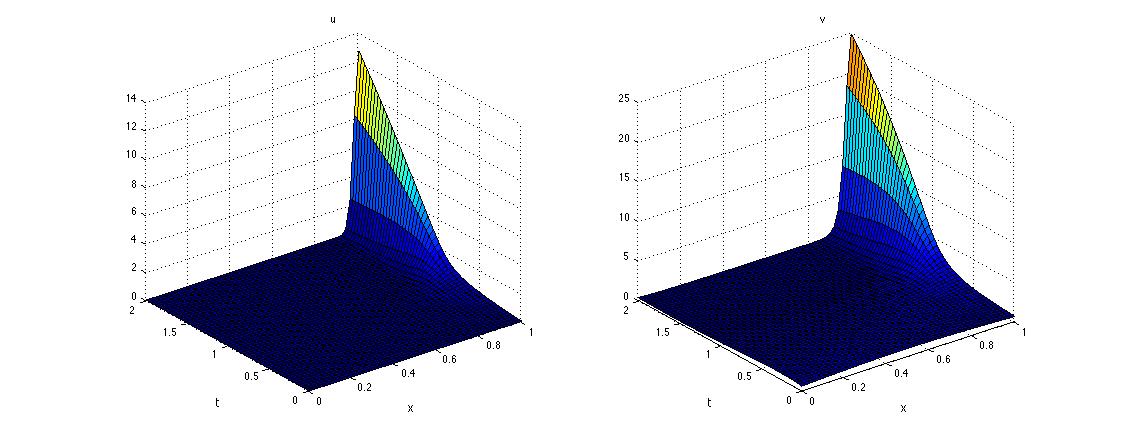}
\caption{This shows the migration of the cell concentrations for $u$ and $v$ to the $x=1$ boundary. Coefficient values were chosen arbitrarily, with $D_{1}=D_{2}=0.1, D_{3}=1, D_{4}=7$, $\alpha_{1}=\alpha_{2}=10$, $\mu+0.001, \mu_{a}=\mu_{b}=1$, $\beta_{1}=1, \beta_{2}=10$. The proliferation term $\rho(b)f(v)$ was replaced with $0.001\;v$, and boundary conditions were linearized, e.g. $-v_{x}(0,t)=b(0,t)$.}
\end{figure}
\section{Further Remarks on Solutions, a Formal Perturbation,  and Numerical Solutions}
Since the full arterial plaque model, from which the system analyzed here, is concerned with the development and possible degradation of a fibrous cap associated with the plaque, our interest in the qualitative behavior of our solution is to have $u$ and $v$ aggregate near the boundary $x=1$ as $t$ gets large. We indicate such behavior via a formal perturbation argument and a numerical simulation, but first we consider the dynamics of the average concentrations.
\subsection{Dynamics of the average concentration}
Let $\bar{u}(t)=\int_{0}^{1}u(x,t)dx$, and similarly for the variables $v, a, b$. Then, from averaging the equations in (\ref{sys}) and using the boundary conditions (\ref{bcs}), we have ($\bar{u}_0=\int_{0}^{1}u(x,0)dx$, etc.) we obtain
\begin{equation*}
\begin{array}{l}
\dt \bar u (t) = g(t) - \mu \bar u (t)\\
\dt \bar a (t) = \beta_1 \bar v (t) -\mu_a \bar a(t)\\
\dt \bar b(t) = \beta_2 \bar u (t) - \mu_b \bar b(t)\\
\dt \bar v = D_2h(b(0,t)) + \int_0^1 \rho(b(x,t))f(v(x,t))dx.
\end{array}
\end{equation*}
Observe first that
\[
\bar{u}(t)=\bar{u}_{0}e^{-\mu t}+D_{1}\int_{0}^{t}e^{-\mu(t-s)}g(s)ds.
\]
Due to
\begin{equation*}
\int_{0}^{t}e^{\mu s}g(s)ds \leq \frac{e^{\mu t}}{\sqrt{2\mu}}||g||_{L^{2}(0,T)} \quad \quad t \in (0,T),
\end{equation*}
we have
\begin{equation*}
\bar{u}(t) \leq \bar{u}_{0}e^{-\mu t} + \frac{D_1}{\sqrt{2\mu}}||g||_{L^{2}(0,T)} \leq \bar{u}_{0} + \frac{D_1}{\sqrt{2\mu}}||g||_{L^{2}(0,T)}.
\end{equation*}
This, if $\int_0^\infty g^2 < \infty$, we have the time invariant bound
\begin{equation*}
\bar{u}(t) \leq \bar{u}_{0} + \frac{D_1}{\sqrt{2\mu}}||g||_{L^{2}(0,\infty)}, t >0.
\end{equation*}
Now denoting by $\tu$ the difference of the averages of two solutions, we obtain
\[
\dt \tu (t)= -\mu \tu(t),
\]
which in turn yields
\begin{gather}  \label{uave}
\tu (t) = e^{-\mu t}\tu(0).
\end{gather}
In particular, together with the time invariant bound obtained above, this means $\bar u(t) \ra \bar u_c$ for some $0 \le \bar u_c < \infty$.

Similarly, using the fact that $\bar u(t)$ is uniformly bounded in time, it readily follows that
$\sup_{t\ge 0} \bar b(t) < \infty$. Moreover, denoting $\tb$ the difference of the average of two solutions, we have
\[
\dt \tb (t) + \mu_b \tb (t) = \beta_2 \tu (t).
\]
Using \eqref{uave}, we readily obtain an explicit solution 
\[
\tb = e^{-\mu_bt}\left[\tb(0) + \frac{\beta_2 \tu(0)}{\mu - \mu_b}\right] -  
\frac{\beta_2 \tu(0)}{\mu - \mu_b}e^{-\mu t}.
\]
This immediately yields that $\tb(t) \ra 0$ as $t \ra \infty$. Together with the time uniform bound for $\bar b(t)$, it follows as before that $\bar b (t) \ra \bar b_c$ as $t \ra \infty$, where $0 \le \bar b_c < \infty$.

\comments{
\begin{equation*} \begin{array}{l}
\bar{u}(t)=\bar{u}_{0}e^{-\mu t}+D_{1}\int_{0}^{t}e^{-\mu(t-s)}g(s)ds \\ \\
\bar{a}(t)=\bar{a}_{0}e^{-\mu_{a}t}+\beta_{1}\int_{0}^{t}e^{-\mu_{a}(t-s)}\bar{v}(s)ds \\ \\
\bar{b}(t)=\bar{b}_{0}e^{-\mu_{b}t}+\beta_{2}\int_{0}^{t}e^{-\mu_{b}(t-s)}\bar{u}(s)ds
\end{array}
\end{equation*}
along with
\begin{equation*}
\frac{d\bar{v}}{dt}(t)=D_{2}h(b(0,t))+\int_{0}^{1}\rho(b(x,t))f(v(x,t))dx \;.
\end{equation*}
Now
\begin{equation*}
\int_{0}^{t}e^{\mu s}g(s)ds \leq \frac{e^{\mu t}}{\sqrt{2\mu}}||g||_{L^{2}(0,T)} \quad \quad t \in (0,T)
\end{equation*}
so 
\begin{equation*}
\bar{u}(t) \leq \bar{u}_{0}e^{-\mu t} + \frac{D_1}{\sqrt{2\mu}}||g||_{L^{2}(0,T)} \leq \bar{u}_{0} + \frac{D_1}{\sqrt{2\mu}}||g||_L^{2}(0,T)
\end{equation*}
for all $t \in (0,T)$. Substituting this into the $\bar{b}$ equation gives 
\begin{gather*}
\bar{b}(t) \leq \bar{b}_{0}e^{-\mu_{b}t}+ \frac{\beta_{2}D_{1}}{\mu_{b}\sqrt{2\mu}}||g||_{L^{2}(0,T)}\left (1-e^{-\mu_{b}t}\right ) + \beta_{2}\bar{u}_{0}\left \{ \begin{array}{ll} 
te^{-\mu_{b}t}  & \mbox{if $\mu_{b}=\mu$} \\ \\
\frac{e^{-\mu t}-e^{-\mu_{b}t}}{\mu_{b}-\mu}  &  \mbox{if $\mu_{b} \neq \mu$}
\end{array} \right. \\ \\
\leq \bar{b}_{0}+ \frac{\beta_{2}D_{1}}{\mu_{b}\sqrt{2\mu}}||g||_{L^{2}(0,T)} + \frac{\beta_{2}\bar{u}_{0}}{\mu_{b}} \;\;.
\end{gather*}
}

For the $\bar{a}$ and $\bar{v}$ case, let $h_M>0$ and $\rho_{M}>0$ be such that $h(b)\leq h_M$ and $\rho(b)\leq \rho_{M}$ for $b \geq 0$ and denote $f_M = \sup_v |f(v)| < \infty $ (we used the assumption that $f$ is bounded).  Then
\begin{equation*} \begin{array}{ll}
\bar{v}(t)\leq \bar{v}_{0} + (D_2h_M + \rho_Mf_M)t  & \mbox{for $t \in (0,T)$, and} \\ \\
\bar{a}(t) \leq \bar{a}_{0}+  \frac12 (D_2h_M + \rho_Mf_M)t^2 & \mbox{for all $t \in (0,T)$.}
\end{array}
\end{equation*}

The above arguments lead us to the intriguing possibility that the averages $\bar a(t)$ and $\bar v(t)$ can increase polynomially in $t$ while $\bar u, \bar b$ converge to a finite number  as $t \ra \infty$. A detailed analysis of the asymptotic behavior of the entire system in relation to the physical (biological) parameters and its biological consequences will be explored in a future work.

\subsection{A Formal Perturbation}
Now we indicate in an informal way that the variables do what they are supposed to do physically, that is, they ``stack'' up next to the boundary $x=1$. For this we scale problem (\ref{sys}) and consider the long time behavior as a perturbation problem, and then show a numerical solution to the system. For purposes of the calculations below we make a few more assumptions. First, we consider the original spatial domain as $0<x<L$, which provides a convenient length scale. Next, we consider the `linearized' problem in that we assume $f(v)=\rho_{0}v$, and let $\rho$ be a constant. Let $(U, V, A, B)$ be characteristic values for $(u, v, a, b)$, so $\tilde{u}=u/U$, etc., and define dimensionless variables $\tilde{x}=x/L$, $\tilde{t}=t/T$. Assume $D_{4}=D_{3}$, and for some constant $k>0$, $D_{1}=D_{2}=kD_{3}$. Let $T=L^{2}/D_{3}$,  $\Delta=\alpha_{1} A/kD_{3}$, $\Gamma=\alpha_{2} B/kD_{3}$, $\tilde{\mu}_{a}=\mu_{a}T$, $\tilde{\mu}_{b}=\mu_{b}T$, $G_{1}=\beta_{1}TV/A$, $G_{2}=\beta_{2}TU/B$. These parameters will be considered order one. We also consider the case where $k$ is sufficiently large compared to $\mu$ or $\rho$ (``slow'' cell migration in the intima), that the scaled $\mu$ and $\rho$ are small; that is, $\varepsilon \tilde{\mu}=\mu T$, and $\varepsilon \tilde{\rho}=\rho T$, where $0<\varepsilon <<1$. Then, considering only the steady state case (and dropping the tilde notation), we have the system
\begin{equation} \left \{\begin{array}{ll} \label{ss}
0=u_{xx}-\Delta(ua_{x})_{x}-\varepsilon \mu u  & \quad 0 < x < 1\;,\, t>0 \\ \\
0=v_{xx}-\Gamma(vb_{x})_{x}+\varepsilon \rho v \\ \\
0=a_{xx} + G_{1}v - \mu_{a}a \\ \\
0=b_{xx} + G_{2}u - \mu_{b}b 
\end{array} \right.
\end{equation}
with boundary conditions
\begin{equation} \left \{ \begin{array}{l} 
u_{x}(0)=a_{x}(0)=b_{x}(0)=0\;,\; v_{x}(0)=h(b(0)) \\ \\
u_{x}(1)=Q\;,\; v_{x}(1)=a_{x}(1)=b_{x}(1)=0
\end{array} \right.
\end{equation}
where we assume $\lim_{t \to \infty}g(t)=Q$ exists. This is a regular perturbation problem, so the solution is analytic in $\varepsilon$. Writing $u \sim u_{0}+\varepsilon u_{1} + \ldots$, with analogue expansions for $v, a, b$, the order one problem is 
\begin{equation} \left \{\begin{array}{ll} \label{ss1}
u_{0}^{''}-\Delta(u_{0}a_{0}^{'})^{'} =0 & \quad 0 < x < 1\;,\, t>0 \\ \\
v_{0}^{''}-\Gamma(v_{0}b_{0}^{'})^{'} =0\\ \\
a_{0}^{''} + G_{1}v_{0} - \mu_{a}a_{0}=0 \\ \\
b_{0}^{''} + G_{2}u_{0} - \mu_{b}b_{0}=0 
\end{array} \right.
\end{equation}
with boundary conditions
\begin{equation} \left \{ \begin{array}{l} 
u_{0}^{'}(0)=a_{0}^{'}(0)=b_{0}^{'}(0)=0\;,\; v_{0}^{'}(0)=h(b_{0}(0)) \\ \\
u_{0}^{'}(1)=Q\;,\; v_{0}^{'}(1)=a_{0}^{'}(1)=b_{0}^{'}(1)=0\;.
\end{array} \right.
\end{equation}
One way of writing the solution is 
\begin{equation*} \begin{array}{l}
u_{0}(x)=u_{0}(0)e^{\Delta (a_{0}(x)-a_{0}(0))} \\ \\
v_{0}(x)e^{h_{0}\int_{0}^{x}\frac{dz}{v_{0}(z)}}=v_{0}(0)e^{\Gamma (b_{0}(x)-b_{0}(0))} \\ \\
a_{0}(x)=G_{1}\int_{0}^{1}G(x,s)v_{0}(s)ds \\ \\
b_{0}(x)=G_{2}\int_{0}^{1}\mathcal{G}(x,s)u_{0}(s)ds 
\end{array}
\end{equation*}
where $h_{0}=h(b_{0}(0))$, which is really a functional of $a_{0}$. The $G$ and $\mathcal{G}$ are appropriate Green's functions for the $a$ and $b$ problems. For example, with $\eta=\sqrt{\mu_{a}}$, 
\begin{equation*}
G(x,s)=\left \{ \begin{array}{ll}\frac{\cosh(\eta s)-\tanh(\eta)\sinh(\eta s)}{\eta \tanh(\eta)(\sinh^{2}(\eta s)-\sinh(\eta s)+1)}\;\cosh(\eta x) & x<s \\ \\
\frac{\cosh(\eta x)-\tanh(\eta)\sinh(\eta x)}{\eta \tanh(\eta)(\sinh^{2}(\eta x)-\sinh(\eta x)+1)}\;\cosh(\eta s) & x>s
\end{array} \right.
\end{equation*}
($\mathcal{G}$ is the same except $\eta=\sqrt{\mu_{b}}$.) We do not have specific values for $u(0), v(0), a(0), b(0)$ or those at $x=1$, but we can make some qualitative comments on the solution at this level. 
Assume that the principle qualitative behavior of the solution to (\ref{ss}) is given by the solution of the lowest order problem, and with positivity of $u(0)$ and $v(0)$, and knowing the formulas for the Green's functions $G(x,s)$ and $\mathcal{G}(x,s)$, we know $u_{0}(x) >0, u_{0}^{'}(x)>0, u_{0}^{''}(x)>0$ on $(0,1)$. Also, $v_{0}(x)>0$, and $a_{0}(x)>0, a_{0}^{'}(x)>0, a_{0}^{''}(x)>0$. Again, $b_{0}(x)>0, b_{0}^{'}(x)>0, b_{0}^{''}(x)>0$, and therefore, $v_{0}^{'}(x)>0, v_{0}^{''}(x)>0$. This is all consistent with these cells and chemicals piling up at the boundary $x=1$. 

If the flux of $u(x,t)$ at $x=1$ becomes zero for $t$ greater than some $T>0$, then $Q=0$. It is likely that $b(0)$ is negligible, and since $h(0)=0$, then these conditions given homogeneous Neumann boundary conditions for system (\ref{ss}). For purposes of this next calculation, assume small chemical decay terms; that is, replace parameters $\mu_{a}, \mu_{b}$ with $\varepsilon \mu_{a}, \varepsilon \mu_{b}$. Then, instead of problem (\ref{ss1}), we have for the lowest order problem
\begin{equation} \left \{\begin{array}{ll} \label{ss0}
u_{0}^{''}-\Delta(u_{0}a_{0}^{'})^{'} =0 & 0 < x < 1\;,\, t>0 \\ \\
v_{0}^{''}-\Gamma(v_{0}b_{0}^{'})^{'} =0\\ \\
a_{0}^{''} + G_{1}v_{0} =0 \\ \\
b_{0}^{''} + G_{2}u_{0} =0 \;\;.
\end{array} \right.
\end{equation}

With our Neumann boundary conditions,  $u_{0}^{'}-\Delta u_{0}a_{0}^{'}=0$, so $u_{0}=Ce^{\Delta a_{0}}$, for some constant $C$. Similarly, $v_{0}=De^{\Gamma b_{0}}$ for some constant $D$. Substituting into (\ref{ss0}), we have 
\begin{equation} \begin{array}{l}
a_{0}^{''}+G_{1}De^{\Gamma b_{0}}=0\;, \quad \quad a_{0}^{'}(0)=a_{0}^{'}(1)=0 \\ \\
b_{0}^{''}+G_{2}Ce^{\Delta a_{0}}=0\;\;, \;\quad \quad b_{0}^{'}(0)=b_{0}^{'}(1)=0
\end{array}
\end{equation}
Again assume $\varepsilon <<1$ so all the interesting steady state behavior is contained in the lowest order terms. For convenience drop the ``0'' subscript notation. Write 
\begin{equation*} \begin{array}{l}
a^{''}e^{\Delta a}+G_{1}De^{\Delta a +\Gamma b}=0=b^{''}e^{\Gamma b}+G_{2}Ce^{\Delta a +\Gamma b}\;, \quad \mbox{ hence} \\ \\
G_{2}Ce^{\Delta a +\Gamma b}=-b^{''}e^{\Gamma b}=-\frac{G_{2}C}{G_{1}D}a^{''}e^{\Delta a} \;.
\end{array}
\end{equation*}
That is, $b^{''}e^{\Gamma b}=Ka^{''}e^{\Delta a}$, where $K=G_{2}C/G_{1}D$. If we think of $a$ and $b$ as independent variables, then this expression should be constant. Say $b^{''}e^{\Gamma b}=R$, or $b^{''}-Re^{-\Gamma b}=0$. Thus, $b^{''}b^{'}-Rb^{'}e^{-\Gamma b}=0$, which implies $\frac{1}{2}(b_{x})^{2}+\frac{R}{\Gamma}e^{-\Gamma b}=P$, which we assume is a positive constant. Therefore, $b^{'}=\pm \sqrt{2(P-(R/\Gamma)e^{-\Gamma b})}$, or 
\begin{equation*} \begin{array}{l}
\int \frac{db}{\sqrt{P-\frac{R}{\Gamma}e^{-\Gamma b}}}=\pm \sqrt{2}(x+x_{0})\;,\; \quad \mbox{so} \\ \\
\int \frac{e^{\Gamma b/2}db}{\sqrt{e^{\Gamma b}-\frac{R}{\Gamma P}}}=\pm \sqrt{2P}(x+x_{0})
\end{array}
\end{equation*}
Let $z=e^{\Gamma b/2}\;(>1)$, then 
\begin{equation*}
\pm \sqrt{2P}(x+x_{0})=\frac{2}{\Gamma}\int \frac{dz}{\sqrt{z^{2}-N}}\;,
\end{equation*}
where $N:=R/\Gamma P$. This gives $\pm \sqrt{2P}(\Gamma/2)(x+x_{0})=\pm \ln[z\pm\sqrt{z^{2}-N}]$, or $e^{\pm \sqrt{P/2}\Gamma(x+x_{0})}=z\pm \sqrt{z^{2}-N}$. If we let $e^{- \sqrt{P/2}\Gamma(x+x_{0})}=z-\sqrt{z^{2}-N}<z+\sqrt{z^{2}-N}=e^{ \sqrt{P/2}\Gamma(x+x_{0})}$, then
\begin{equation*}
e^{\Gamma b/2}=z=\frac{z+\sqrt{z^{2}-N}+z-\sqrt{z^{2}-N}}{2}=\cosh[\Gamma \sqrt{\frac{P}{2}}(x+x_{0})]\;,
\end{equation*}
so
\begin{equation*}
b(x)=\frac{2}{\Gamma}\ln(\cosh[\Gamma \sqrt{\frac{P}{2}}(x+x_{0})])\;.
\end{equation*}
Similarly, 
\begin{equation*}
\frac{1}{2}(a_{x})^{2}+\frac{R}{k\Delta}e^{-\Delta a}=K_{1} \;\;\Rightarrow \;\; \int \frac{e^{\Delta a/2}da}{\sqrt{e^{\Delta a}-(R/K\Delta K_{1})}}=\pm \sqrt{2K_{1}}(x+x_{0})\;.
\end{equation*}
This leads to 
\begin{equation*}
a(x)=\frac{2}{\Delta}\ln(\cosh[\Delta \sqrt{\frac{K_{1}}{2}}(x+x_{0})])\;.
\end{equation*}
Note that $b(x)>0, b^{'}(x)>0$, and $b^{''}(x)>0$ (and the same with $a(x)$), so $a$ and $b$ have the correct qualitative graph for accumulating its concentration at $x=1$.

Now, given those expressions for $a=a_{0}(x)$ and $b=b_{0}(x)$, we have 
\begin{equation*}
u_{0}(x)=C\cosh^{2}[\Delta \sqrt{\frac{K_{1}}{2}}(x+x_{0})]\;, \quad \quad v_{0}(x)=D\cosh^{2}[\Gamma \sqrt{\frac{K_{1}}{2}}(x+x_{0})]\;.
\end{equation*}
(If $x_{0}=0$, then $C=u_{0}(0), D=v_{0}(0)$.) Note that $u_{0}^{'}(x)>0, u_{0}^{''}(x)>0$ (and the same for $v_{0}$), so $u_{0}, v_{0}$ have the same qualitative graph as $a_{0}$ and $b_{0}$ have, with their maximum at $x=1$.

Figure 1 also shows the concentrations of $u$ and $v$ evolve to the boundary $x=1$ that represent the boundary between the intima layer of the artery, and the lumen (central blood flow region). 

\section{Discussion}
The present analysis of this cross-chemotaxis system is the core of a more extensive model system for an arterial (atherosclerotic) plaque model, whose modeling goal is to determine a quantitative understanding of plaque vulnerability mechanisms. However, the system analyzed here should have independent interest since it has potentially a richer pattern formation structure, than the single cell, single chemical (``auto-chemotaxis'') case, particularly in higher space dimensions. If, upon scaling our model, some parameters are relatively small compared to other parameters, this would open opportunity to do some perturbation analysis on the system. 

A particular modeling assumption we made was to keep the chemoattractant terms of the simplest Keller-Segel kind. In our biological situation, the chemicals tend to bind to cell surface receptors, and a more general form of chemoattractant mechanism might be more appropriate. For example, we could replace the $\alpha_{1}ua_{x}$ (respectively, $\alpha_{2}vb_{x}$) term by $\alpha_{1}uB_{1}(a)a_{x}$ (respectively, $\alpha_{2}vB_{2}(b)b_{x}$). At the present time we do not have the data to determine what form the $B_{1}(\cdot)$ and $B_{2}(\cdot)$ would take. 

In the full arterial plaque model some accounting of fluid shear stress on the endothelial cell layer separating the blood from the plaque will be considered. This is because differences in shear stress trigger different chemical pathways in the endothelial cells that have a measurable effect on chemical interactions at the boundary of the plaque. In the simplest scenario concerning system (\ref{sys}), accounting for this external signal would require modification of the boundary conditions for $a$ and $b$ at $x=1$ to be non-homogeneous flux conditions depending on the strength of the signal, which in turn leads to Robin boundary conditions for $u$ and $v$ at $x=1$. We have not explored this modification yet.   

Another aspect of arterial plaque modeling is that many more chemo-attractants are present in the plaque, along with other cell types. In a more detailed model formulation, foam cells and T-lymphocytes, for example, should be accounted for. These, along with endothelial cells, and not just smooth muscle cells and macrophages, produce chemicals that affect the migration, motility, and proliferation of various cell populations. A similar observation can be made in various inflammation studies and in wound healing modeling. Hence, we have formulated a model system involving an arbitrary number of cell populations and an arbitrary number of chemical attractants, and we will analyze the system in a future paper. 

Last, but not least, plaques are three dimensional, and exploration of such cross-chemoattractant problems discussed here needs to be carried over to higher dimensions when the questions being asked needs such a generalization.

\section{References}
\begin{enumerate}
\item Amann, H, Dynamic theory of quasilinear parabolic equations II. Reaction-diffusion systems, Diff. and Int. Equations 3(1)(1990), 113-75.
\item Amann, H, Dynamic theory of quasilinear parabolic equations III. Global existence, Math. Z. 202(1989), 219-250.
\item Biler, P, Local and global solvability of some parabolic systems modeling chemotaxis, Adv. Math. Sci. Appl. 8(1998), 715-743.
\item Gajewski, H and Zacharias, K, Global behavior of a reaction-diffusion system modeling chemotaxis, Math. Nachr. 195(1998), 77-114.
\item Hillen, T, and Painter, KJ, A user's guide to PDE models for chemotaxis, J. Math. Biol. 58(2009), 183-217.
\item  Hillen T, and Potapov A, The one-dimensional chemotaxis model: global existence and asymptotic profile, Math. Meth. Appl. Sci. \textbf{27}(2004), 1783-1801.
\item Horstmann, D, Lyapunov functions and $L^p$-estimates for a class of reaction-diffusion systems, Colloq. Math. 87(1)(2001), 113-127.
\item Horstmann, D, From 1970 until present: the Keller-Segel model in chemotaxis and its consequences, I. Jahresberichte DMV 105(3)(2003), 103-165.
\item Jager, W and Luckhaus, S, On explosions of solutions to a system of partial differential equations modeling chemotaxis, Trans. AMS 329(2)(1992), 819-=824.
\item Lauffenburger, D, Aris, R and Keller, K, Effects of cell motility and chemotaxis on microbial population growth, Biophys. J. 40(1982), 209-219.
\item Little, MP, Gola, A and Tzoulaki, I, A model of cardiovascular disease giving a plausible mechanism for the effect of fractionated low-dose ionizing radiation exposure, PloS Comput. Biol. 5(10)(2009), 1-11.
\item Osaki, K and Yagi, A, Finite dimensional attractor for one-dimensional Keller-Segel equations, Funkcial. Ekvac. 44(2001), 441-469.
\item Osaki, K, Tsujikawa, T, Yagi, A and Mimura, M, Exponential attractor for a chemotaxis-growth system of equations, Nonlinear Analysis 51(2002), 119-144.
\item Painter, KJ and Hillen, T, Volume-filling and quorum-sensing in models for chemosensitive movement, Can. Appl. Math. Quart. 10(4)(2002), 501-543.
\item Perthame, B, \textit{Transport Equations in Biology}, Birkhauser, Basel, 2007.
\item Suzuki, T, \textit{Free Energy and Self-Interacting Particles}, Birkhauser, Boston, 2005.
\item Wang, X, Qualitative heavier of solutions of chemotactic diffusion systems: effects of motility and chemotaxis and dynamics, SIAM J. Math Anal 31(3)(2000), 535-560.
\item Wrzosek, D, Global attractor for a chemotaxis model with prevention of overcrowding, Nonlinear Analysis 59(2004), 1293-1310.
\item Zeng, B, Steady states of a chemotaxis system, Appl. Math. 1(1990), 78-83.

\end{enumerate}

\end{document}